\documentclass[graybox]{svmult}

\usepackage{mathptmx}       
\usepackage{helvet}         
\usepackage{courier}        
\usepackage{type1cm}        
\newcommand{\R}{\mathbb{R}}	                
\usepackage{makeidx}         
\usepackage{graphicx}        
\usepackage{multicol}        
\usepackage[bottom]{footmisc}
\usepackage{amssymb}

\usepackage{amsmath,amsfonts}

\makeindex             

\begin{document}
\title*{Conjugate times and regularity of the minimum time function with differential inclusions}
\titlerunning{Conjugate times and regularity of the minimum time function}
\author{Piermarco Cannarsa and Teresa Scarinci}
\institute{Piermarco Cannarsa \at Dipartimento di Matematica,
Universit\`a di Roma Tor Vergata, Via della Ricerca Scientifica 1, 00133 Roma, Italy, \email{cannarsa@mat.uniroma2.it}
\and Teresa Scarinci \at Dipartimento di Matematica,
Universit\`a di Roma Tor Vergata, Via della Ricerca Scientifica 1, 00133 Roma, Italy and CNRS, IMJ-PRG, UMR 7586, Sorbonne Universit\'es, UPMC Univ Paris 06, Univ Paris Diderot, Sorbonne Paris Cit\'e, Case 247, 4 Place Jussieu, 75252 Paris, France, \email{teresa.scarinci@gmail.com}}
%
%
\maketitle

\abstract*{Each chapter should be preceded by an abstract (10--15 lines long) that summarizes the content. The abstract will appear \textit{online} at \url{www.SpringerLink.com} and be available with unrestricted access. This allows unregistered users to read the abstract as a teaser for the complete chapter. As a general rule the abstracts will not appear in the printed version of your book unless it is the style of your particular book or that of the series to which your book belongs.
Please use the 'starred' version of the new Springer \texttt{abstract} command for typesetting the text of the online abstracts (cf. source file of this chapter template \texttt{abstract}) and include them with the source files of your manuscript. Use the plain \texttt{abstract} command if the abstract is also to appear in the printed version of the book.}

\abstract{This paper studies the regularity of the minimum time function, $T(\cdot)$, for a control system with a closed target, taking the state equation in the form of a differential inclusion. Our first result is a sensitivity relation which guarantees the propagation of the proximal subdifferential of $T$ along any optimal trajectory. Then, we obtain the local $C^2$ regularity of the minimum time function along optimal trajectories by using such a relation to exclude the presence of conjugate times.}
\section*{Introduction}
This paper aims to refine the study of the regularity properties of the value function of the  time optimal control problem in nonparameterized  form, that is, when the state equation is given as a differential inclusion. This problem seems hard to address by parametrization techniques, as it has been observed in the  recent papers \cite{MR2918253}, \cite{MR2873196}, \cite{MR3005035}, and \cite{MR3318195}.

Recall the minimum time problem 
$\mathcal P(x)$ 
consists of minimizing the time $T$ over all trajectories of a controlled dynamical system that originate from an initial point $x\in\R^n$ and terminate on a compact target set 
$\mathcal K\subseteq\R^n$.  Specifically, the problem $\mathcal P(x)$ is
\begin{equation}\label{Cost}
\min T,
\end{equation}
where the minimization is over all absolutely continuous arcs $y(\cdot)$ defined on an interval $[0,T]$ that satisfy the differential inclusion
\begin{equation}\label{intro:DI}
\begin{cases}
\dot y(t)\in F\bigl(y(t)\bigr) \quad\text{a.e. }t\in[0,T]\\
y(0)=x
\end{cases}
\end{equation}
and the terminal condition
$y(T)\in\mathcal K.
$ Here, $F:\R^n\rightrightarrows\R^n$ is a Lipschitz continuous multifunction having a sublinear growth such that the associated Hamiltonian  
\begin{equation*}
H(x,p)=\sup_{v\in F(x)} \langle - v , p \rangle\qquad(x,p)\in \R^n\times\R^n
\end{equation*}
is semiconvex in $x$ and differentiable in $p$, whenever $p\neq 0$. 
The minimum time function, $T(x)$, is defined as the optimal value in \eqref{Cost}. 

The main object of our analysis  are sensitivity relations, that is, inclusions that identify the dual arc as a suitable generalized gradient of the minimum time function $T(\cdot)$, evaluated along a given minimizing trajectory. The importance of  such relations is well acknowledged and will be made clear by the applications we provide to the differentiability of $T(\cdot)$.

Sensitivity relations have a long history dating back, at least, to the papers \cite{Clarke:1987:RMP:35498.35509}, \cite{MR923279}, \cite{Cannarsa:1991:COT:120771.118946}, and \cite {MR894990} that  studied  optimal control problems of Bolza type with finite time horizon. In \cite{MR1780579}, such relations were adapted to  the minimum time problem for the parameterized control system   
\begin{equation}\label{intro:DE}
\dot{y}(t)=f(y(t),u(t)) \quad t \geq 0
\end{equation}
assuming that:
\begin{description}
\item[(i)] $\mathcal K$ has the inner sphere property, and
\item[(ii)]  Petrov's controllability condition is satisfied on $\partial \mathcal K$.
\end{description}
For any optimal trajectory $y(\cdot)$ of \eqref{intro:DE} originating at a point $x$ in the controllable set, the result of \cite{MR1780579} ensures the existence of an arc $p$, called a dual arc, such that: 
\begin{itemize}
\item $(y,p)$ satisfies the Hamiltonian system
\begin{equation}\label{intro:CJ}
\left\{\begin{array}{rll}
-\dot{y}(t)&=& \nabla_p H (y(t),p(t))\\
\dot{p}(t)&\in & \partial_x^- H (y(t),p(t))
\end{array}\right.
\quad 
0\le t\le T(x)=:T
\end{equation}
together  with the {\em transversality} condition  
\begin{equation*}
p(T) =-\dfrac{\nu}{H(y(T),-\nu)} ,
\end{equation*}
where $\nu$ is any unit {\em inner} normal to $\mathcal K$ at $y(T)$;
\item $p(t)$ belongs to the {\em Fr\'echet} superdifferential of $T(\cdot)$ at $y(t)$ for all $t\in [0,T(x))$.
\end{itemize}
In  \cite{MR3005035}, the above result was extended to nonparameterized control  systems by developing an entirely different proof, based on the Pontryagin maximum principle rather than  linearization techniques as in \cite{MR1780579}. 
In \cite{MR3318195},  assumption (ii) above was removed,  still keeping (i) in force, showing that $p(t)$ is either a {\em proximal} or a {\em horizontal} supergradient of $T(\cdot)$ at $y(t)$,  for all $t\in [0,T(x))$,  depending on whether Petrov's condition is satisfied  or not at $y(T(x))$. 

With respect to sensitivity relations, the purpose of the present paper is 
to derive analogous inclusions for the the  {\em proximal subdifferential} of the minimum time function; more precisely, we will prove  the propagation  of the subdifferential  of  $T(\cdot)$ along optimal trajectories.
By `propagation of the proximal subdifferential' we mean the fact that, if a proximal subgradient of $T(\cdot)$ exists at some point $x$ of the reachable set---so that the minimum time function is differentiable at $x$---and $y(\cdot)$ is a time optimal trajectory starting at $x$, then $p(t)$ belongs to the proximal subgradient of $T(\cdot)$ at $y(t)$  for all $t\in [0,T(x))$.
Such an invariance of the subdifferential with respect to the Hamiltonian flow associated with \eqref{intro:CJ} was pointed out in \cite{frankowska:hal-00851752} for functionals in the calculus of variations and \cite{Cannarsa2013791} for optimal control problems of Bolza type. A similar result was obtained in \cite{Nostro} for the Mayer problem and in \cite{frankowska:hal-00981788} for the minimum time problem for parameterized control system. In Theorem \ref{Lemma_sub_prox:ch3} of this paper, we show that such a property holds for the minimum time problem with a state equation in the form of a  differential inclusion.

We give two applications of the above sensitivity relations. The first one (Theorem~\ref{differentiability} below) ensures that the differentiability of $T(\cdot)$ propagates along an optimal trajectory, $y(\cdot)$, originating at a point $x$ of the controllable set if and only if Petrov's condition is satisfied at $y(T(x))$. This property follows directly from the above relations which guarantee that the corresponding dual arc is contained in both Fr\'echet semidifferentials whenever $T(\cdot)$ is differentiable at $x$ . Our second application concerns the local smoothness of the minimum time function along an optimal trajectory $y(\cdot)$, that is, 
the property of having continuous second order derivatives in a neighborhood of $\{y(t)~:~0\le t<T(x)\}$. In Theorem \ref{TheoLocalReg:ch3}, we show that this is indeed the case whenever $T(\cdot)$ has a proximal subgradient at the starting point of $y(\cdot)$.

In order to prove the local smoothness of $T(\cdot)$ along an optimal trajectory we need to analyze conjugate times, and give sufficient conditions to exclude the presence of such times. The notion of conjugate point is classical in the calculus of variations and optimal control. 
Recently, conjugate times have been considered in \cite{frankowska:hal-00981788} linearizing the system on the whole $\R^n$ but neglecting the role of the time variable. In such a paper, the degeneracy condition is assigned on the tangent space to the target, which is an $(n-1)$-dimensional space and the authors show that the absence of conjugate times at a point $x$ ensures the $C^1$-smoothness of $T(\cdot)$ along the trajectory originating at $x$. In this paper, we return to the `classical' definition of conjugate point and formulate a sufficient condition for smoothness in terms of conjugate times (see Theorem~\ref{TheoremCT}), much in the spirit of the result of \cite{MR1916603}. 

In this way we  deduce that, if the proximal subgradient of  $T(\cdot)$ is nonempty at some point $x$, then the minimum time function is locally smooth along the optimal trajectory originating at $x$.

The  paper is organized as follows.  Background material is collected in Section \ref{Chapter3:Notation}. In Section \ref{Chapter3:preliminary}, we recall preliminary results and discuss the main assumptions we work with. 
Section \ref{sec:CP} is devoted to the analysis of conjugate times. Section \ref{Chapter3:results} contains our sensitivity relations and their applications to regularity.

%
%
%

\section{Notation}\label{Chapter3:Notation}
Let us fix the notation and list some basic facts. Further details can be found in several books, for instance \cite{MR1048347,MR709590,MR2662630,MR2041617}.\\
  
We denote by $|\cdot|$ the Euclidean norm in $\mathbb{R}^n$ and by
$\langle \cdot,\cdot \rangle$ the inner product.
$B(x,\epsilon)$ is
the closed ball of radius $\epsilon > 0$ centered at
$x$, and $S^{n-1}$ the unit sphere in $\mathbb{R}^n$. $ \mathbb{R}^{n\times n}$ is the set of $n\times n$ real matrices and $\parallel Q \parallel$
is the operator norm of a matrix $Q$, $Q^*$ is the transpose of $Q$, $\ker Q$ is the kernel of $Q$, while $I_n$ is the $n\times n$ identity
 matrix. Recall that $\parallel Q \parallel = \sup\{ |\langle Ax,x \rangle | : x \in S^{n-1} \}$  
 for any symmetric $n\times n$ real matrix $Q$. Moreover, co$\,E$, $\partial E$, $\overline{E}$ and $E^C$ are the convex hull, the boundary, the closure and the complement of a set $E\subset\mathbb{R}^n$, respectively.\\
 
Let $K$ be a closed subset of $\mathbb{R}^n$ and $x \in K$. $N_K^C (x)$ denotes the Clarke normal cone to $K$ at $x$.  A vector $v \in\mathbb{R}$ is a \emph{proximal (outer) normal} to $K$ at $x$, and we write $v \in N^P_K
(x)$, if there exists $\sigma=\sigma(x,v)$ such that, for all $y\in K$,
\begin{equation}\label{SI}
\langle v, y - x \rangle \leq \sigma |y-x|^2.
\end{equation}

In $K$ is a convex subset of $\mathbb{R}^n$, the proximal normal cone to $K$ at $x$ coincides with the convex normal cone to $K$ at $x$.\\
A vector $\nu$ is a proximal inner normal to $K$ at $x$ if $\nu \in N^P_{K^C}(x)$. We call it an unit proximal inner normal to $K$ at $x$ if it belongs to the set $N^P_{K^C}(x)\cap S^{n-1}$.\\
We say that $K$ satisfies the \emph{inner sphere property} of radius $R$, $R>0$, if for every $x\in\partial K$ there exists a nonzero vector $\nu_x \in N^P_{K^C}(x)$ such that \eqref{SI} holds true with $\sigma = |\nu_x| (2R)^{-1}$ and $v=\nu_x$ and for all $y\in K^C$. Equivalently, for all $x\in\partial K$ there exists a vector $0\neq \nu_x\in N_{K^C}^P (x)$ \emph{realized by a ball} of radius $R$, that is,
\[
B\left( x + R \frac{\nu_x }{\mid \nu_x \mid},R \right) \subset K .
\]
Roughly speaking, if $K$ satisfies the inner sphere property of radius $R$ then we have an upper bound for the curvature of $\partial K$, even though $\partial K$ may be a nonsmooth set. Indeed, any $x\in\partial K$ belongs to some closed ball $y_x +R B(0,1)\subset K$. This fact suggests that, in some sense, the curvature of $\partial K$ is bounded above and excludes the presence of outward pointing corners on $\partial K$.\\

If $f : [t_0,t_1] \rightarrow \mathbb{R}^n$ is
continuous, $f\in C([t_0,t_1])$, define $\|f\|_{\infty} = \max_{t\in[t_0,t_1]} |f(t)|$.
Moreover, we usually refer to an absolutely continuous function $x:[t_0,t_1] \rightarrow \mathbb{R}^n$ as an arc. The space $C^k(\Omega)$,
 where $\Omega$ is an open subset of $\mathbb{R}^n$, is the space of all functions that are continuously differentiable $k$ times on $\Omega$.
The gradient of $f$ is $\nabla f(\cdot)$, if it does exist. Moreover, if $f$ is twice differentiable at some $x\in \Omega$, then  $\nabla^2 f(x)$ denotes the Hessian of $f$ at $x$. Let $f:\Omega \rightarrow \mathbb{R}$ be any real-valued function
defined on an open set $\Omega\subset\mathbb{R}^n$. Let $x\in \Omega$ and $p\in\mathbb{R}^n$. We say that:
\begin{itemize}
\item $p$ is a \emph{Fr\'echet subgradient} of $f$ at $x$, $p\in\partial^- f(x)$, if 
$$ \liminf_{y \rightarrow x} \frac{f(y)-f(x)-\langle p, y - x\rangle  }{\mid y - x \mid}\geq 0 ,$$
\item $p$ is a \emph{Fr\'echet supergradient} of $f$ at $x$, $p\in\partial^+ f(x)$, if 
$$\limsup_{y \rightarrow x} \frac{f(y)-f(x)-\langle p, y - x\rangle  }{\mid y - x \mid}\leq 0,$$
\item $p$ is a \emph{proximal subgradient} of $f$ at $x$, $p\in\partial^{-,P} f(x)$, if $\exists\; c,\;\rho \geq 0$ such that
$$ f(y)-f(x)-\langle p, y-x \rangle \geq - c \vert y-x \vert^2,\; \forall y \in B(x,\rho),$$
\item $p$ is a \emph{proximal supergradient} of $f$ at $x$, $p\in\partial^{+,P} f(x)$, if $\exists\; c,\;\rho \geq 0$ such that
$$ f(y)-f(x)-\langle p, y-x \rangle \leq  c \vert y-x \vert^2,\; \forall y \in B(x,\rho).$$
\end{itemize}
If $f$ is Lipschitz, $\zeta\in\mathbb{R}^n$ is a
\emph{reachable gradient} of $f$ at $x\in \Omega$ if there exists
a sequence $\lbrace x_j \rbrace \subset \Omega$ converging to $x$
such that $f$ is differentiable at $x_j$ for all $j\in
\mathbb{N}$ and $ \zeta = \lim_{ j \rightarrow \infty} \nabla
f(x_j).$ Let $\partial^* f(x)$ denote the set of all reachable
gradients of $f$ at $x$. The \emph{(Clarke) generalized gradient}
of $f$ at $x\in\Omega$, $\partial f(x)$, is the set
co$\, (\partial^* f(x))$.\\
For an open set $\Omega\subset \mathbb{R}^n$, $f :
\Omega\rightarrow \mathbb{R} $ is \emph{semiconcave} if it is
continuous in $\Omega$ and there exists a constant $c$ such that
$$ f(x + h) + f(x - h) - 2 f(x) \leq c | h |^2, $$ for all $x, h\in
\mathbb{R}^n$ such that $[ x - h, x + h] \subset \Omega$. We say
that a function $f$ is semiconvex on $\Omega$ if and only if $-f$
is semiconcave on $\Omega$. We recall below some properties of
semiconcave functions (for further details see, for instance,
\cite{MR2041617}).
\begin{proposition}
Let $\Omega\subset \mathbb{R}^n$ be  open, $f : \Omega \rightarrow
\mathbb{R}$ be a semiconcave function with semiconcavity constant
$c$, and let $x \in \Omega$. Then, $f$ is locally Lipschitz on
$\Omega$ and the following holds true
\begin{enumerate}
\item $p\in \mathbb{R}^n$ belongs to $\partial^+ f(x)$ if and only
if, for any $y \in \Omega$ such that $[y, x] \subset \Omega$,
\begin{equation}\label{Booo:ch3}
f(y) - f(x) - \langle p, y - x \rangle  \leq c |y - x|^2.
\end{equation}
\item $\partial f(x) = \partial^{+} f(x)=co\ ( \partial^{\ast} f(x))$.
\item If $\partial^+ f(x)$ is a singleton, then $f$ is differentiable at $x$.
\end{enumerate}
\end{proposition}
If $f$ is semiconvex, then (\ref{Booo:ch3}) holds reversing the inequality and the sign of the quadratic term, and the other two statements are true with the subdifferential instead of the superdifferential. \\

Let $M \subset \mathbb{R}^n $ be a \emph{$C^m$-manifold} of dimension $n-1$ and fix $\xi_0\in M$. Let $A\subset\mathbb{R}^{n-1}$ be an open set, let $\phi:A \rightarrow \mathbb{R}^{n}$ be a map of class $C^m$ such that $\phi(A) \subset M$, $D\phi(y)$ has rank equal to $n-1$ for all $y\in A$ and $\phi(\eta_0)=\xi_0$ for some $\eta_0\in A$. We call $\phi$ a local \emph{parameterization} of $M$ near $\xi_0.$ The components $(\eta_1,...\eta_{n-1})$ of a point $\eta=\phi^{-1}(\xi)\in A$ are usually called \emph{local coordinates} of $\xi\in M$.\\
An application $F:M \rightarrow \mathbb{R}^n$ is of class $C^k$ at $\xi_0 \in M$ if the map $F\circ \phi^{-1}:\phi(A) \rightarrow\mathbb{R}^n$ is of class $C^k$ at $\nu_0:=\phi^{-1} (\xi_0)$ for any local parameterization $\phi$ of $M$ near $\xi_0$. Equivalently, $F: M \rightarrow \mathbb{R}^n$ is of class $C^k$ at $\xi_0$ if there exists a local parameterization $\phi$ of $M$ near $\xi_0$ such that $F\circ \phi^{-1}$ is of class $C^k$ at $\eta_0:=\phi^{-1} (\xi_0)$.

\section{Assumptions and preliminary results}\label{Chapter3:preliminary}
The \emph{minimum time problem} $\mathcal{P}(x)$ consists of minimizing the time $T$ over all trajectories of a differential inclusion that start from an initial point $x\in\R^n$ and reach a nonempty compact set $\mathcal{K} \subseteq\R^n$, usually called \emph{target}. Specifically, for any absolutely continuous function $y^x(\cdot)\in AC([0,+\infty);\mathbb{R}^n)$ that solves the differential inclusion
\begin{equation}\label{DI}
\begin{cases}
\dot y(t)\in F\bigl(y(t)\bigr) \quad\text{a.e. }t\geq 0 \\
y(0)=x,
\end{cases}
\end{equation}
let us denote by 
$$\theta(y^x(\cdot)):=\inf \lbrace t\geq 0 :~ y^x(t)\in \mathcal{K} \rbrace$$ 
the first time at which the trajectory $y^x(\cdot)$ reaches the target $\mathcal{K}$ starting from $x$. By convention, we set $\theta(y^x(\cdot))=+\infty$ whenever $y^x(\cdot)$ does not reach $\mathcal{K}$. Here and throughout the paper, $F:\R^n\rightrightarrows\R^n$ is a given multifunction that satisfies the so-called \emph{Standing Hypotheses}:
\begin{description}
\item[(SH)]\quad
$\begin{cases}
{\bf 1)} \, F(x) \text{ is nonempty, convex, and compact for each }x\in \R^n, \\
{\bf 2)} \,  F\text{ is locally Lipschitz with respect to the Hausdorff metric}, \\
{\bf 3)} \, \text{there exists }\rho>0 \text{ so that }\max\{|v|:v\in F(x)\}\leq \rho(1+|x|).
\end{cases}$
\end{description}
The \emph{minimum time} function $T:\mathbb{R}^n\rightarrow [0,+\infty]$ is defined by: for all $x\in\mathbb{R}^n$, 
\begin{equation}\label{MINIMUM}
T(x):= \inf \lbrace \theta(y^x(\cdot)):~ y^x(\cdot) \mbox{ solves } \eqref{DI}  \rbrace.
\end{equation}  
$T(x)$ represents the minimum time needed to steer  the point $x$ to the target $\mathcal{K}$ along the trajectories of \eqref{DI}.  It is well-known that  $(SH)$ guarantees the existence of absolutely continuous solutions to \eqref{DI} defined on $[0,+\infty)$. Moreover, if $x$ is in the \emph{reachable set} $\mathcal{R}$ (i.e. $T(x)<+ \infty$) then $\mathcal{P}(x)$ has an optimal solution, that is, a solution to  \eqref{DI} that gives the minimum in \eqref{MINIMUM}. 
The main assumptions of this paper are expressed in terms of the Hamiltonian $H:\R^n\times\R^n\to\R$ associated to $F$, that is, the function defined by
\begin{equation}\label{Hamiltonian}
H(x,p)=\sup_{v\in F(x)} \langle - v , p \rangle.
\end{equation} 
We shall suppose that
$$ (H)
\left\{ 
\begin{array}{ll}
&\mbox{for  every } r>0 \\
&(i)\; \exists\ c \geq 0 \mbox{ so that }, \forall p\in S^{n-1}, x \mapsto H(x,p)\mbox{ is semiconvex on } B(0,r) \mbox{ with}\\
&\;\;\;\;\;\; \mbox{constant }  c_r, \\
&(ii)\; \nabla_p H(x,p) \mbox{ exists and is  Lipschitz continuous in } x \mbox{ on } B(0,r), \mbox{ uniformly } \\
&\;\;\;\;\;\; \mbox{for }p\in S^{n-1}.
\end{array}\right.$$

We recall that $(H)$ was introduced for the minimum time problem in \cite{MR2918253} to derive sufficient conditions for the semiconcavity of the minimum time function. We refer the reader to \cite{MR2728465,Nostro,SecondoNostro} for a detailed discussion of $(H)$. 
\begin{remark} 
Actually, in \cite{Nostro,MR2728465} the authors suppose that the Hamiltonian $H^+(x,p):= \sup_{v\in F(x)}\langle v,p\rangle$ satisfies $(H)$. On the other hand, it is easy to compute that $H^{+}(x,p)=H(x,-p)$, and so $H^+$ satisfies $(H)$ if and only if so does $H$.
\end{remark}
We recall below a classical result known as \emph{Maximum principle} for  the minimum time problem. It yields as necessary condition for the optimality of a trajectory $x(\cdot)$ the existence of a \emph{dual arc} $p(\cdot)$ such that the pair $(x,p)$ satisfies an Hamiltonian inclusion and a transversality condition.
\begin{theorem}\label{NC}
Assume that $(SH)$ and $(H)$ hold. Suppose $x(\cdot)$ is an optimal solution of the minimum time problem $\mathcal{P}(x)$, reaching the target $\mathcal{K}$ at time $T:=T(x)$. Then there exists an absolutely continuous arc $p : [0, T] \rightarrow \mathbb{R}^n$, $p(\cdot)\neq 0$, such that for a.e. $t\in [0,T]$,
\begin{equation}\label{CJ:ch3}
\left\{\begin{array}{rll}
-\dot{x}(t)&=& \nabla_p H (x(t),p(t)),\\
\dot{p}(t)&\in & \partial_x^- H (x(t),p(t)),
\end{array}\right.
\quad p(T) \in N_{\mathcal{K}}^C(x(T)). 
\end{equation}
\end{theorem}
The classical formulation of the above theorem (see, for instance, \cite{MR709590}) is expressed in terms of the ``complete'' Hamiltonian system $(\dot{x},\dot{p}) \in \partial H(x,p)$ (where $\partial H$ stays for  Clarke's generalized gradient of $H$ in $(x,p)$). However, the ``splitting Lemma'' in \cite{SecondoNostro} (Lemma $2.9$) guarantees  that under our assumptions these two formulations are  equivalent. 
\begin{remark}\label{RemarkDualArc:ch3} Let us give two remarks.
\begin{itemize}
\item[(a)]\ Let $(x,p)$ be a solution to the Hamiltonian inclusion  
\begin{equation}\label{CJM}
\left\{\begin{array}{rll}
-\dot{x}(t)&\in& \partial_p^- H (x(t),p(t)),\\
\dot{p}(t)&\in & \partial_x^- H (x(t),p(t)),
\end{array}\right.
\quad \mbox{a.e. in } [t_0,T].
\end{equation}
Then, there are only two possible cases:
\begin{itemize}
\item either $p(t)\neq 0$ for all $t\in [t_0,T]$,
\item or $p(t)=0$ for all $t\in [t_0,T]$.
\end{itemize}
Indeed, consider $r>0$ such that $x([t_0,T]) \subset
B(0,r)$. If we denote by $c_r$ a Lipschitz constant for $F$ on
$B(0,r)$, then $c_r \vert p\vert$ is a Lipschitz constant for
$H(\cdot, p)$ on $B(0,r)$. Thus,
\begin{equation*}
\vert \zeta \vert \leq c_r \vert p \vert \quad \forall \zeta \in
\partial^-_x H(x,p),~ \forall x \in B(0,r),~ \forall p \in \mathbb{R}^n.
\end{equation*}
Hence,  $ \vert
\dot{p}(s)\vert\leq c_r \vert p(s)\vert$ for
a.e. $s \in [t_0,T]$. Therefore, Gronwall's Lemma allows to conclude.
\item[(b)]\ If $(x,p)$ is a solution to \eqref{CJM}, then for any $\lambda >0$ the pair $(x,\lambda p)$ solves \eqref{CJM} as well. Indeed, by the positive 1-homogeneity in $p$ of the Hamiltonian, that is $H(x,\lambda p)=\lambda H(x,p)$ for all $\lambda>0$, $x$, $p\in\mathbb{R}^n$, it follows that $\partial_x H(x,\lambda p)=\lambda \partial_x H(x,p)$ and $\partial_p H(x,\lambda p)=\partial_p H(x,p)$ for all $\lambda>0$, $x$, $p\in\mathbb{R}^n$. Thus, the proof of our claim is an easy verification.
\end{itemize}
\end{remark}
For our aims, sometimes we shall need more refined necessary conditions than the ones in Theorem \ref{NC}. 
Assuming the interior sphere property on the target $\mathcal{K}$ allows to further specify the transversality condition.
\begin{proposition}\label{NC++}
Assume that $(SH)$ and $(H)$ hold. Suppose $x(\cdot)$ is an optimal solution for the minimum time problem $\mathcal{P}(x)$, reaching the target $\mathcal{K}$ at time $T:=T(x)$, and that there exists $0\neq \nu \in N_{\mathcal{K}^C}^P ( x(T))$ realized by a ball of radius $R$, that is,
\[
B\left( x(T) + R \frac{\nu }{\mid \nu \mid},R \right) \subset \mathcal{K}.
\]
Then there exists an absolutely continuous arc $p : [0, T] \rightarrow \mathbb{R}^n$, $p(\cdot)\neq 0$, such that for a.e. $t\in [0,T]$,
\begin{equation}
\left\{\begin{array}{rll}
-\dot{x}(t)&=& \nabla_p H (x(t),p(t)),\\
\dot{p}(t)&\in & \partial_x H (x(t),p(t)),
\end{array}\right.
\quad p(T)= -\nu . 
\end{equation}
\end{proposition}
\begin{proof}
The trajectory $x(\cdot)$ is time-optimal even for the problem obtained replacing the target $\mathcal{K}$ by the ball $B_1:=B\left( x(T) + R \nu \mid \nu \mid^{-1},R \right)$. Moreover, $N_{B_1}^C(x(T))=\lbrace - \nu \rbrace$.  Thus, applying Theorem \ref{NC} to this new problem we prove our claim.
\end{proof}

\begin{remark}\label{WithPetrov}
If we suppose in addition that 
\begin{equation}\label{Petrov}
\mu(-\nu):= H(x(T),-\nu)^{-1} > 0, 
\end{equation}
then the above theorem together with Remark \ref{RemarkDualArc:ch3} (b) gives that there exists an absolutely continuous arc $p : [0, T] \rightarrow \mathbb{R}^n$, $p(\cdot)\neq 0$, such that $(x,p)$ solves, for a.e. $t\in [0,T]$,
\begin{equation}
\left\{\begin{array}{rll}
-\dot{x}(t)&=& \nabla_p H (x(t),p(t)),\\
\dot{p}(t)&\in & \partial_x H (x(t),p(t)),
\end{array}\right.
\quad p(T)=-\mu(-\nu)\nu . 
\end{equation} 

\end{remark}

In addition to our  assumptions on $F$ and $H$, further hypotheses on the target set $\mathcal{K}$ might be needed, such as the inner sphere property and the so-called \emph{Petrov condition} we recall below:
\begin{description}
\item[(PC)] $\exists \delta>0$   such that  $H(x,\zeta )\geq \delta\|\zeta\|$  for all  $x\in\partial \mathcal{K}$ and all $\zeta\in N^P_{\mathcal{K}}
(x)$.
\end{description}
Assumption $(PC)$ turns out to be equivalent to the Lipschitz continuity of the minimum time function $T(\cdot)$ in a neighborhood of $\mathcal{K}$. It is also necessary for the semiconcavity of $T(\cdot)$ up to a boundary of $\mathcal{K}$ and equivalent to the validity of a bound of $T$ in terms of the distance function from the target $\mathcal{K}$, which is defined as
\[ d_{\mathcal{K}}:\mathbb{R}^n\rightarrow \mathbb{R}^+,\quad  d_{\mathcal{K}}(x):= \inf \lbrace \mid y-x\mid : y\in \mathcal{K} \rbrace. 
\] 
Recall, among the other things, that assuming  Petrov's condition on the target $\mathcal{K}$ guarantees that \eqref{Petrov} always holds true. For a comprehensive treatment and  further references  on this subject we refer to the book \cite{MR2041617}.
In sections \ref{sec:CP} and \ref{sub:last} we shall also assume  that $\mathcal{K}$ is the closure of its interior and
\begin{description}
\item[(A)] $\partial \mathcal{K}$ is an $(n-1)$-dimensional manifold of class $C^2$.
\end{description}
Whenever $(A)$ holds true, $\mathcal{K}$ satisfies the inner sphere property with a uniform positive radius. Moreover, the signed distance  from the target $\mathcal{K}$, that is,
$$ b_{\mathcal{K}}:\mathbb{R}^n \rightarrow \mathbb{R},\quad b_{\mathcal{K}}(\cdot):= d_{\mathcal{K}}(\cdot)-d_{\mathcal{K}^C}(\cdot),$$ 
is a function of class $C^2$ in a neighborhood of $\partial \mathcal{K}$, and $- \nabla b_{\mathcal{K}}(\xi)$ is a proximal inner normal to $\mathcal{K}$ at $\xi\in\partial \mathcal{K}$ with unit norm.
\section{Conjugate times for the minimum time problem}\label{sec:CP}
The aim of this section is to extend the main result in \cite{MR1344204} to the minimum time problem. More precisely, we show that the absence of conjugate times is equivalent to the propagation of the local regularity of the minimum time function. Let us mention that a partial result in this framework has been recently given in \cite{ frankowska:hal-00981788}. On the other hand, our notion of conjugate time is more in the spirit of \cite{MR1916603} and allows to recover a stronger result than the one in \cite{frankowska:hal-00981788}.
\subsection{Conjugate times for the minimum time problem}
In this section, we assume $(SH)$, $(PC)$, and $(A)$ and suppose that the Hamiltonian $H$ is of class $C^2(\mathbb{R}^n \times (\mathbb{R}^n \setminus \lbrace 0 \rbrace))$. Given $\xi\in\partial \mathcal{K}$, set $g(\xi):=\mu(\xi)\nabla b_{\mathcal{K}}(\xi) $, where $\mu(\xi)$ is the positive constant $H(\xi,\nabla b_{\mathcal{K}}(\xi))^{-1}$. Recall also that, thanks to $(A)$, the function $g$ is of class $C^1$ in a neighborhood of $\partial\mathcal{K}$. Therefore, we denote by $(Y(\xi,\cdot),P(\xi,\cdot))$ (or, briefly, by $(Y(\cdot),P(\cdot))$) the solution of the \emph{backward Hamiltonian system}
\begin{equation}\label{CJ*:ch3}
\left\{\begin{array}{rllrrl}
\dot{Y}(t)&=& \nabla_p H (Y(t),P(t)), & \quad Y(0)&=&\xi, \\
-\dot{P}(t)&=& \nabla_x H (Y(t),P(t)), & \quad P(0)&=&g(\xi).
\end{array}\right.
\end{equation}
We recall that for any $\xi\in\partial\mathcal{K}$ the solution $(Y(\cdot),P(\cdot))$ to \eqref{CJ*:ch3} is well-defined on $[0,+\infty)$ and the functions $Y,~P$ are of class $C^1$ with respect to $\xi$ and the time in $\partial\mathcal{K}\times [0,+\infty)$ (for the proof of these facts see, for instance, Section $3$ in \cite{MR1916603}).\\
Since $\partial\mathcal{K}$ is a $C^2$-manifold of dimension $n-1$, for any $\xi_0\in \partial \mathcal{K}$ there exist a $C^2$ local parameterization of $\partial\mathcal{K}$:
\[ \phi: A\subset \mathbb{R}^{n-1} \rightarrow \mathbb{R}^n, \quad \eta \rightarrow \phi(\eta)=\xi.
\]
Set $\eta_0:=\phi^{-1}(\xi_0)$. Let us denote by $Y_{\xi,t}(\xi,t)$  and $P_{\xi,t}(\xi,t)$ the Jacobians of $Y(\phi(\cdot),\cdot)$ and $P(\phi(\cdot),\cdot)$ with respect to the state variable $\eta\in \mathbb{R}^{n-1}$ and time, that is, 
\[
Y_{\xi,t}(\xi,t)=Y_{\eta,s}(\phi(\eta),t), \quad
P_{\xi,t}(\xi,t)=P_{\eta,s}( \phi(\eta),t).
\]
Therefore, note that $Y_{\xi,t}(\xi,t)$ and $P_{\xi,t}(\xi,t)$ belong to $\mathbb{R}^{n\times n}$ and the pair $(Y_{\xi,t},P_{\xi,t})$ 
solves the \emph{variational system}
\begin{equation}\label{CP_prem:ch3}
\left\{\begin{array}{rllrrl}
\dot{Y}_{\xi,t}&=&H _{xp} (Y,P)Y_{\xi,t}+ H_{pp}(Y,P) P_{\xi,t} ,& Y_{\xi,t}(\xi,0) &=& \left( \frac{\partial \phi}{\partial \eta}(\eta), \nabla_p H(\xi,p)\right), \\
-\dot{P}_{\xi,t}&=& H_{xx}(Y,P) Y_{\xi,t}+ H_{px} (Y,P) P_{\xi,t},&  P_{\xi,t}(\xi,0)&=&\left( \frac{\partial g }{\partial \eta} (\phi(\eta) ),-\nabla_x H(\xi,p)\right) ,
\end{array}\right.
\end{equation}
where we have set $p:=\mu(\xi)\nabla b_{\mathcal{K}} (\xi)$. Matrix $Y_{\xi,t}(\xi,0)$ is invertible; indeed, by $(PC)$ and the choice of $\mu(\cdot)$ it follows that 
\[
0\neq \mu(\xi)^{-1} H(\xi,\mu(\xi)\nabla b_{\mathcal{K}}(\xi) )= \langle \nabla_p H(\xi,\mu(\xi)\nabla b_{\mathcal{K}} (\xi)), \nabla b_{\mathcal{K}} (\xi) \rangle .
\]
Thus, the vector $\mu(\xi)\nabla b_{\mathcal{K}} (\xi)$ is non-characteristic for the data $g(\cdot)$, that is,
\[ \langle \nabla_p H(\xi,\mu(\xi)\nabla b_{\mathcal{K}} (\xi)), \nabla b_{\mathcal{K}} (\xi) \rangle \neq 0. \]
It is natural to introduce the following definition of \textit{conjugate time}.
\begin{definition}\label{conjugate}
Let $\xi_0 \in \partial \mathcal{K}$ and let $\phi$ a local $C^2$ parameterization of $\partial\mathcal{K}$ near $\xi_0$. Let $(Y_{\xi,t},P_{\xi,t})$ be the solution to \eqref{CP_prem:ch3}. Define
\[
\overline{t}=\sup \lbrace t\in [0,+\infty) :\; \det Y_{\xi,t}(\xi_0,s)\neq 0 \;
 \mbox{ for all } s\in [0,t] \rbrace.
 \]
The time $\overline{t}$ is called \emph{conjugate} for $\xi_0$  if $\overline{t}<+\infty$. 
\end{definition}
Thus, if $\overline{t}$ is conjugate for $\xi_0$ then $\det Y_{\xi,t}(\xi_0,\overline{t} )=0$.\\
Note that the solution $(Y_{\xi,t},P_{\xi,t})$ to \eqref{CP_prem:ch3} depends on the parameterization $\phi$. On the other hand, the ranks of the values of the maps $Y_{\xi,t}(\xi_0,\cdot)$ and $P_{\xi,t}(\xi_0,\cdot)$ are independent of the particular choice of $\phi$, as well as the above definition of conjugate time.\\
By standard techniques one deduces that if $\det Y_{\xi,t} (\xi_0,t)\neq 0$, then there exists a neighborhood of $(\xi_0,t)$ in $\partial \mathcal{K}\times \mathbb{R}$ such that the matrix  $Y_{\xi,t} (\xi,s)$ is nonsingular for any vector $(\xi,s)$ in such a neighborhood. Furthermore, if there are no conjugate times for $\xi_0$ on some interval $[0,a]$, then the map $Y(\cdot,\cdot)$ provides a diffeomorphism from a neighborhood $ J_{\xi_0} \times U_t$ of $(\xi_0,t)$ in $\partial \mathcal{K}\times \mathbb{R}$ onto its image for all $t\in [0,a]$. 
Moreover, it is easy to check that the function $R(\xi_0,t):=P_{\xi,t}(\xi_0,t)Y^{-1}_{\xi,t}(\xi_0,t)$, as long as $Y_{\xi,t}(\xi_0,t)$ is invertible, solves the Riccati equation
\begin{equation}\label{Riccati}
\left\lbrace\begin{array}{l}
\dot{R}+ H_{px}(Y,P) R + R H_{xp}(Y,P)+R H_{pp}(Y,P) R+ H_{xx}(Y,P)=0,\\
R(\xi_0, 0)= P_{\xi,t} (\xi_0,0)Y^{-1}_{\xi,t}(\xi_0,0) .
\end{array}
\right.
\end{equation}
For a fixed $\theta\in\mathbb{R}^n \setminus \lbrace 0 \rbrace$ and for any $t>0$, let us denote by $w(t)$ the $2 n$-vector given by $(Y_{\xi,t}(\xi_0,t )\theta,P_{\xi,t}(\xi_0,t )\theta)$. It is easy to check that $w(\cdot)$ solves a linear differential system with nonzero initial data, since $Y_{\xi,t}(\xi_0,0)$ has rank $n$. By well-known properties of linear systems, it follows that $w(t)\neq 0$ for all $t>0$. This means that  for any $\theta \in\mathbb{R}^n\smallsetminus \lbrace 0 \rbrace$ and $t>0$,
$$ Y_{\xi,t}(\xi_0,t )\theta =0\; \Rightarrow\; P_{\xi,t}(\xi_0,t )\theta \neq 0.$$ 
Therefore, it is easy to understand that a time $\overline{t}$ is conjugate for $\xi_0$ if and only if $[0,\overline{t})$ is the maximal interval of existence of the solution $R(\xi_0,\cdot)$  to \eqref{Riccati} and $\overline{t}<+\infty$. Thus, $\overline{t}$ is a finite blow-up time for $R(\xi_0,\cdot)$, that is,
\[
\lim_{t\nearrow \overline{t}} \parallel R (\xi_0,t)\parallel = +\infty.
\]
If $(Y, P)$ is given on a finite time interval $[0,T]$, then the above definition of conjugate time can be adapted, by saying that $\overline{t} \in [0,T]$ is a conjugate time for
$\xi_0$ if and only if $\det Y_{\xi,t}(\xi_0,t)\neq 0$ for all  $t \in [0,\overline{t})$ and  $\det Y_{\xi,t}(\xi_0,\overline{t})=0$. Equivalently, $\overline{t}$ is a conjugate time for
$\xi_0$ if and only if $R_{\xi,t}(\xi_0,\cdot)$ is well defined on $[0,\overline{t})$ and $\lim_{t
\nearrow \overline{t}} \parallel R_{\xi,t}(\xi_0,t) \parallel = + \infty$. 


\subsection{Local regularity of the minimum time function and conjugate times}
Let $\xi_0\in\partial \mathcal{K}$ and $t\geq 0$. Given an open neighborhood $V_{\xi_0}\times I_t$ of $(\xi_0,t)$ in $\partial \mathcal{K} \times\mathbb{R}$, define the set
$$M (V_{\xi_0}\times U_t):= \lbrace (Y(\xi,s),P(\xi,s)) |\ (Y,P) \mbox{ solves } \eqref{CJ*:ch3} \mbox{ with } \xi \in V_{\xi_0},\ s\in U_t \rbrace. $$
When $t> 0$, the set $U_t$ may be viewed as an interval of the form $(t-b,t+b)$ for some $b>0$ and when $t=0$ as the interval $(0,b)$.
\begin{theorem}\label{TheoremCT}
Let us assume $(SH)$, $(PC)$ and $(A)$ and suppose that the Hamiltonian $H$ is of class $C^2(\mathbb{R}^n \times (\mathbb{R}^n \setminus \lbrace 0 \rbrace))$. Fix $\overline{t}>0$. Then, the following two statements are equivalent:
\begin{itemize}
\item[(i)]  for all $t\in[0,\overline{t}]$, there exists an open neighborhood $V_{\xi_0}\times I_t$ of $(\xi_0,t)$ in $\partial \mathcal{K} \times\mathbb{R}$ such that the set
\begin{equation}\label{domain}
\mathcal{D}(V_{\xi_0}\times I_t):= \lbrace Y(\xi,s)| (Y,P) \mbox{ solves } \eqref{CJ*:ch3}, \mbox{ with } \xi \in V_{\xi_0},\ s\in I_t \rbrace 
\end{equation}
is an open subset of $\mathbb{R}^n$, and $M(V_{\xi_0}\times I_t)$ is the graph of a $C^1$ function on $\mathcal{D}(V_{\xi_0}\times U_t)$;
\item[(ii)]\  there are no conjugate times for $\xi_0$ on $[0,\overline{t}]$.
\end{itemize}
\end{theorem} 
In \cite{MR1344204}, Caroff and Frankowska analysed the link between conjugate points and regularity of the value function $V$ for Bolza optimal control problems, showing that the first emergence of a conjugate point corresponds to the first time when $V$ stops to be locally smooth along optimal trajectories. In Theorem \ref{TheoremCT} we prove that the same kind of result holds true also for the minimum time problem. We note that our result cannot be deduced from the one in \cite{MR1344204}---even though the technique of proof is similar---because the definition of conjugate time we use in this paper is different from the one therein. For this reason, we give below the proof of the implication $(ii) \Rightarrow (i)$, which is the one needed to derive Theorem~\ref{TheoLocalReg:ch3} below. 
\begin{proof}
Suppose that there are no conjugate times for $\xi_0$ on $[0,\overline{t}]$. We want to show that there exists a neighborhood $V_{\xi_0}\times I_t$ of $(\xi_0,t)$ in $\partial \mathcal{K} \times\mathbb{R}$ such that $M (V_{\xi_0}\times U_t)$ is a graph of a $C^1$ function on $\mathcal{D}(V_{\xi_0}\times I_t)$, for all $t\in[0,\overline{t}]$. Actually, we shall prove, first, that $M (V_{\xi_0}\times U_t)$ is a graph of a Lipschitz function with Lipschitz constant uniform in $[0,\overline{t}]$. So, proceeding by contradiction, let us fix any neighborhood $V_{\xi_0}\times I_t$ of $(\xi_0,t)$ in $\partial \mathcal{K} \times\mathbb{R}$ and 
let us consider the compact set $\Pi_{t}:= \overline{M (V_{\xi_0}\times I_t)}$ for all $t> 0$. 
It is a well-known fact that there exists a time $t^*>0$ such that $\Pi_t$ is a graph of a Lipschitz function for all $t\in [0,t^*]$. 
Let $a =\sup \mathcal{T}$, where
\begin{align*}
\mathcal{T}:= \lbrace &t\in [0,\overline{t}] :\ \exists k_t \geq 0 \mbox{ s.t. } \Pi_s \mbox{ is a graph of a } k_t \mbox{-Lipschitz function}\\
&\Phi_s:\overline{\mathcal{D} ( V_{\xi_0}\times I_s)}\rightarrow \mathbb{R}^n \mbox{ for all } s\in [0,t] \rbrace. 
\end{align*}
Aiming to a contradiction, suppose that $a\not\in \mathcal{T}$, i.e., $\Pi_a$ is not the graph of a $k$-Lipschitz function. Then, fix $t\in[0,a) $. Since $\det Y_{\xi,t}(\xi_0,t)\neq 0$, without loss of generality, we can suppose that for any vector $(\xi,s) \in \overline{V_{\xi_0}}\times \overline{U_t}$ we have that $\det Y_{\xi,t}(\xi,s)\neq 0$. Moreover, $Y(\cdot,\cdot): V_{\xi_0}\times U_t \rightarrow \mathbb{R}^n $ is an injective continuous map. Thus, $Y(V_{\xi_0}, U_t)$ is an open set by Brouwer's invariance of domain theorem. Note that $\mathcal{D}(V_{\xi_0}\times U_t)\equiv Y(V_{\xi_0}, U_t)$. Consequently, $\mathcal{D} (V_{\xi_0}\times U_t)$ is open and its closure is
\[ \overline{\mathcal{D}(V_{\xi_0}\times U_t)}:= \lbrace Y(\xi,s)| (Y,P) \mbox{ solves } \eqref{CJ*:ch3} \mbox{ with } \xi\in \overline{ V_{\xi_0}},\ s\in \overline{U_t}\rbrace. \]
Note that the map $\Phi_t$ is a.e. differentiable on $\mathcal{D}(V_{\xi_0}\times U_t)$ for all $t\in [0,a)$.
Since $\Pi_a$ is not a Lipschitz graph, there exist two sequences $ t_i\nearrow a$ and $ \lbrace x_i\rbrace_{i\in\mathbb{N}} \subset \mathcal{D} (V_{\xi_0}\times U_{t_i})$ such that 
\[ \parallel D \Phi_{t_i}(x_i) \parallel\rightarrow + \infty. \]
Equivalently, we can find a sequence of vectors $\lbrace u_i,v_i\rbrace_{i\in\mathbb{N}}\subset \mathbb{R}^n \times \mathbb{R}^n$ such that $D \Phi_{t_i}(x_i) u_i = v_i$, $|v_i|=1$ for all $i\in\mathbb{N}$ and $|u_i|\rightarrow 0$ as $i\rightarrow + \infty$.\\
Since  $ \lbrace x_i\rbrace_{i\in\mathbb{N}} \subset \mathcal{D} (V_{\xi_0}\times U_{t_i})$, there exist a sequence of vectors $\lbrace \xi_i\rbrace_i\subset V_{\xi_0} $ and one of times $s_i\in U_{t_i}$ such that the solution $(Y(\xi_i,\cdot),P(\xi_i,\cdot))$ to \eqref{CJ*:ch3} solves $Y(\xi_i,s_i)=x_i$ and $P(\xi_i,s_i)=\Phi_{t_i}( x_i)$. Now, let us consider the linearization of the system associated to $(Y(\xi_i,\cdot),P(\xi_i,\cdot))$ at $(x_i,p_i)$ given by the solution $(w_i,q_i)$ to 
\begin{equation}
\left\{\begin{array}{rllrrl}
\dot{w_i}&=&H _{xp} (Y(\xi_i,t),P(\xi_i,t))w_i + H_{pp}(Y(\xi_i,t),P(\xi_i,t))q_i ,& w_i(s_i) &=& u_i, \\
-\dot{q_i}&=& H_{xx}(Y(\xi_i,t),P(\xi_i,t)) w_i+ H_{px} (Y(\xi_i,t),P(\xi_i,t)) q_i,&  q_i(s_i)&=&v_i.
\end{array}\right.
\end{equation}
Consequently, 
\begin{equation}\label{blow}
D \Phi_{t_i} (x_i) w_i(s_i) = q_i(s_i), \;\; |w_i(s_i)| \rightarrow 0 \mbox{ and } |q_i(s_i)|=1. 
\end{equation}
After possibly passing to a subsequence, we may assume that the sequence $\lbrace \xi_i\rbrace_{i\in\mathbb{N}}$ converges to some vector $\overline{\xi}\in \overline{V_{\xi_0}}$ and $\lbrace s_i\rbrace_{i\in\mathbb{N}}$ to some time $\overline{s}\in \overline{U_t}$, as $i\rightarrow +\infty$. Then, passing to the limit as $i\rightarrow +\infty$, it is easy to deduce from \eqref{blow} that the vector $\overline{\xi}$ has a conjugate time equal to $\overline{s}$, i.e., $\det Y_{\xi,t}(\overline{\xi},\overline{s})=0$. Since $(\overline{\xi},\overline{s})\in \overline{ V_{\xi_0}}\times \overline{U_t}$, we obtain a contradiction. Therefore, $\Pi_t$ is a graph of a Lipschitz function for all $t\in [0,\overline{t}]$. Since $g$ is of class $C^1$ in a neighborhood of $\partial\mathcal{K}$, by well-known properties of linearized systems we deduce that, for every parameterization $\phi$ of $V_{\xi_0}$, $\Phi_t \circ \phi^{-1}$ is of class $C^1$ and (i) holds true.
\end{proof}
\begin{remark}\label{RemarkRapresentation}
Suppose that the map $\Phi_t$ is of class $C^1$ on the set $\mathcal{D}(V_{\xi_0}\times U_t)$ for all $t\in[0,\overline{t}]$. Then, it is easy to understand that its Jacobian is given by: for all $\xi\in V_{\xi_0}$ and $s\in U_t$,
\[ D \Phi_t(Y(\xi,s))= P_{\xi,t}(\xi,s) Y_{\xi,t}(\xi,s)^{-1}, \]
in the sense that the matrix 
\[ P_{\eta,t}(\phi(\eta),s) Y_{\eta,t}(\phi(\eta),s)^{-1} \]
represents the Jacobian of $\Phi_t$ at $Y(\xi,s)$ in the system of local coordinates $(\eta_1,...\eta_{n-1})$ induced by a parameterization $\phi$ of $V_{\xi_0}$.
\end{remark}
A characteristic $Y(\xi,\cdot)$, with $\xi\in\partial\mathcal{K}$, is said to be \emph{optimal} in some interval $[0,\tau]$ if it coincides with an optimal trajectory $y(\cdot)$ starting from $Y(\xi,\tau)$ running backward in time, that is, $Y(\xi,t)=y(\tau-t)$, for all $t\in [0,\tau]$. By the classical method of characteristics, one can deduce that any characteristic $Y(\xi,\cdot)$ is optimal in $[0,\tau^*)$ for some time $\tau^*>0$. Theorem \ref{TheoremCT} allows to deduce that this result holds true as long as there are no conjugate times. 
\begin{corollary}\label{CorollaryOptimality}
Let us assume $(A)$ and suppose that $H$ is of class $C^{2}(\mathbb{R}^n \times (\mathbb{R}^n \times \lbrace 0 \rbrace))$. If there are no conjugate times for $\xi_0$ on the interval $[0,\overline{t}]$, then there exists a neighborhood of $\xi_0$ in $\partial \mathcal{K}$, $V_{\xi_0}$, such that $Y(\xi,\cdot)$ is optimal on $[0,\overline{t}]$ for any $\xi\in V_{\xi_0}$. 
\end{corollary}
\subsection{A characterization of conjugate times}
In this subsection, let us assume $(SH)$, $(PC)$ and $(A)$ and suppose that the Hamiltonian $H$ is of class $C^2(\mathbb{R}^n \times (\mathbb{R}^n \setminus \lbrace 0 \rbrace))$. Let us denote by $Y_{\xi}(\xi,t)$ and $P_{\xi}(\xi,t)$ the Jacobian of $Y(\phi(\cdot),t)$ and $P(\phi(\cdot) ,t)$ with respect to the state variable $\eta\in\mathbb{R}^{n-1}$ evaluated at $(\xi,t)$, that is, 
\[ 
Y_{\xi}(\xi,t)= Y_{\eta} (\phi(\eta),t), \quad 
P_{\xi}(\xi,t)= P_{\eta} (\phi(\eta),t).
\]
One can easily check that the pair $(Y_{\xi}(\xi,\cdot),\ P_{\xi}(\xi,\cdot))$ takes values in $\mathbb{R}^{n \times (n-1)}\times \mathbb{R}^{n \times (n-1)}$ and solves the system
\begin{equation}\label{CPpartial}
\left\{\begin{array}{rllrrl}
\dot{Y}_{\xi}&=&H _{xp} (Y,P)Y_{\xi}+ H_{pp}(Y,P) P_{\xi} ,& Y_{\xi}(\xi,0) &=& \frac{ \partial \phi }{\partial \eta } (\eta), \\
-\dot{P}_{\xi}&=& H_{xx}(Y,P) Y_{\xi}+ H_{px} (Y,P) P_{\xi},&  P_{\xi}(\xi, 0)&=&\frac{\partial g}{ \partial \eta}(\phi( \eta)).
\end{array}\right.
\end{equation}
In the case of a strictly convex Hamiltonian in $p$, the notion of conjugate time can be characterized through the solution of the above system (see \cite[Theorem~6.1]{MR1916603}). Let us now introduce the hypothesis:
\begin{description}
\item[(H2)] The kernel of $H_{pp}(x,p)$ has dimension equal to $1$ for every $(x,p)\in \mathbb{R}^n \times ( \mathbb{R}^n \smallsetminus \lbrace 0 \rbrace ) $, i.e., $\ker H_{pp}(x,p)= p~ \mathbb{R}$.
\end{description} 
We will show that also under the weaker assumption $(H2)$ a similar characterization can be provided.
\begin{proposition}\label{Cristina}
Suppose that $H$ satisfies $(H2)$. 
For any $t>0 $, it holds that $\det Y_{\xi,t}(\xi, t )=0$ if and only if $rk Y_{\xi}(\xi, t ) < n-1$. 
\end{proposition}
To prove the above proposition, we need the following lemma.
\begin{lemma}\label{lemmaTech}
Under the assumptions of Proposition \ref{Cristina}, it holds that, for any $t>0$,
\[ \frac{d}{d s} \det Y_{\xi , s} (\xi,s)\mid_{s=t } \neq 0 \quad \Longleftrightarrow \quad rk Y_{\xi,t}(\xi,t )=n-1.   \]
\end{lemma}
\begin{proof}
First, suppose that $rk Y_{\xi,t}(\xi, t )=n-1$.
Following the same reasoning as in the proof of \cite[Lemma~4.3]{MR1916603}, we have that
\[ \frac{d}{d s} \det Y_{\xi,s}(\xi,s)\mid_{s=t}=  tr  \left( H_{pp}(Y(\xi, t),P(\xi, t)) P_{\xi,t}(\xi, t)Y^{+}_{\xi,t}(\xi, t ) \right),  \]
where $A^+$ denotes the transpose of the matrix of the cofactors of a matrix $A$, that is, $A A^+=A^+A= (det A)I_n$. Moreover, if $\theta$ is such that $\ker Y_{\xi,t} (\xi_0, t )= \theta \mathbb{R}$, then by \cite[Lemma~4.2]{MR1916603} there exists $c>0$ such that
\begin{equation}\label{robavecchia}
\frac{d}{d s} \det Y_{\xi,s}(\xi,s)\mid_{s= t }= c H_{pp}(Y(\xi, t),P(\xi, t)) P_{\xi,t}(\xi,t )\theta\cdot P_{\xi,t}(\xi,t )\theta . 
\end{equation} 
We claim that $P_{\xi,t}(\xi,t )\theta \not\in \ker H_{pp}(Y(\xi, t),P(\xi, t))$. If not, there exists $\lambda\in\mathbb{R}\setminus \lbrace 0\rbrace$ such that $P_{\xi,t}(\xi,t )\theta =\lambda P(\xi,t)$. Now, observe that for all $\xi\in \partial\mathcal{K}$ and all $t\geq 0$ it holds that $ H(Y(\xi,t),P(\xi,t) )=1$. 
Hence, taking the Jacobian of this map at $(\xi,t)$ and recalling that $\ker Y_{\xi,t} (\xi_0, t )= \theta \mathbb{R}$ we obtain that
\[
0= \langle Y_{\xi,t}(\xi, t )\theta, H_x(Y(\xi, t),P(\xi, t)) \rangle + \langle P_{\xi,t}(\xi,t)\theta, H_p(Y(\xi, t),P(\xi, t)) \rangle 
\]
\[ 
= \langle P_{\xi,t}(\xi,t)\theta, H_p(Y(\xi, t),P(\xi, t)) \rangle.
\]
On the other hand, since we are assuming that $P_{\xi,t}(\xi,t )\theta =\lambda P(\xi,t)$, we have
\begin{equation*}
\begin{split}
&\langle P_{\xi,t}(\xi,t)\theta, H_p(Y(\xi, t),P(\xi, t)) \rangle = \langle \lambda P(\xi,t)\theta, H_p(Y(\xi, t),P(\xi, t)) \rangle \\
&=\lambda \langle P(\xi,t)\theta, H_p(Y(\xi, t),P(\xi, t)) \rangle = \lambda H(Y(\xi, t),P(\xi, t))=\lambda,
\end{split}
\end{equation*}
that is in clear contradiction with the equality that is above it. This finally shows that $P_{\xi,t}(\xi,t_0)\theta \not\in \ker H_{pp}(Y(\xi, t),P(\xi, t))$, and so from \eqref{robavecchia} we obtain that $$ \frac{d}{ds} \det Y_{\xi,s}(\xi,s)\mid_{s=t }>0.$$\\
For the proof of the other implication, we refer the reader to the proof of \cite[Lemma~4.3]{MR1916603}.
\end{proof}
\begin{proof}[Proof of Proposition \ref{Cristina}]
It is sufficient to show that if $\det Y_{\xi,t}(\xi, t )=0$ then $rk Y_{\xi}(\xi, t ) < n-1$. Aiming for a contradiction, suppose $\det Y_{\xi,t}(\xi, t )=0$ but $rk Y_{\xi}(\xi, t ) = n-1$. Hence, the vectors $Y_{\eta_i}(\xi,t)$, $i=1,...,n-1$, are linearly independent and, by continuity, there exists $\delta>0$ such that for any time $s\in (t-\delta,t+\delta)$ the vectors $Y_{\eta_i}(\xi_0,t)$, $i=1,...,n-1$, are still linearly independent. We can distinguish to cases:
\begin{itemize}
\item[1.] there exists a sequence of times $t_k \rightarrow t$ as $k \rightarrow \infty$ such that $\det Y_{\xi,t}(\xi, t_k )=0$ for all $k$,
\item[2.] there exists a constant $\delta'\in (0,\delta)$ such that $\det Y_{\xi,t}(\xi, s )\neq 0$ for all $s\in (t-\delta',t+\delta')$.
\end{itemize}
For the discussion of the first case, we refer the reader to the proof of \cite[Theorem 6.1]{MR1916603}. In the second case,  we have that
\[  \frac{d}{d s} \det Y_{\xi , s} (\xi,s)\neq 0 \quad \mbox{for all } s\in (t-\delta',t+\delta')
.\]
Then, Lemma \ref{lemmaTech} implies that $rk Y_{\xi,t} (\xi,t) < n-1$ and this yields the contradiction.
\end{proof}
Under the additional assumption $(H2)$, the above proposition gives an equivalent characterization of conjugate times considering only the spatial Jacobian of the map $Y(\cdot,\cdot)$. More specifically, it follows that a time $\overline{t}$ is conjugate for $\xi_0$ if and only if 
\[
\overline{t}=\sup \lbrace t\in [0,+\infty) :\;  rk Y_{\xi}(\xi_0,s)= n-1 
\mbox{ for all } s\in [0,t] \rbrace,
\] 
and $\overline{t}<+ \infty$. Consequently, $rk Y_{\xi}(\xi_0,\overline{t})< n-1$.
\begin{remark}
Let us suppose that $(H2)$ holds true. If there is no conjugate time for $\xi_0$ on $[0,\overline{t} ]$, then $Y(t,\cdot)$ maps a neighborhood $I_{\xi_0}$ of $\xi_0$ in $ \partial \mathcal{K}$ onto the level sets of the minimum time function, that is, for all $t\in [0,\overline{t} ]$,
$$Y(t,\cdot): I_{\xi_0} \rightarrow \Gamma_ t \subset \mathbb{R}^{n-1},  $$  
where $\Gamma_t := \lbrace x\in\mathbb{R}^n:~ T(x) = t \rbrace$. Moreover, $Y(t,\cdot)$ gives a diffemorphism from a neighborhood of $\xi_0$ in $\partial \mathcal{K}$ onto an open neighborhood of $\Gamma_t$, for any time $t$ smaller than the conjugate time $\overline{t}$.
\end{remark} 

\section[First-order sensitivity relations and applications]{First-order sensitivity relations for the minimum time problem and some applications}\label{Chapter3:results}
The scope of this section is twofold. First, we discuss some sensitivity relations of first order. Subsequently, we apply these results to derive sufficient conditions for the propagation of the regularity of the minimum time function along optimal trajectories.
\subsection{Proximal subdifferentiability of the minimum time function}
The forward propagation of the dual arc into the proximal subdifferential of $T$ is already known for minimum time problems when the dynamic is described by a control system with sufficiently smooth dynamics (see \cite{frankowska:hal-00981788}). We shall extend this result to the differential inclusion case.
\begin{theorem}\label{Lemma_sub_prox:ch3}
Assume $(SH)$ and $(H)$ hold. Let $x_0 \in\mathcal{R}\setminus \mathcal{K}$, let $\overline{x}$ be an optimal trajectory starting from $x_0$  reaching $\mathcal{K}$ in time $T(x_0)$, and $\overline{p}:[0,T(x_0)] \rightarrow \mathbb{R}^n$ be an
 arc such that $(\overline{x},\overline{p})$ solves the system: for a.e. $t\in \left[ 0, T(x_0) \right]$,
\begin{equation*}
\left\{\begin{array}{rllrrl}
-\dot{x}(t) &\in& \partial_p^- H(x(t),p(t)), \quad & x(0)&=&x_0 \\
\dot{p}(t) &\in& \partial_x^- H(x(t),p(t)), \quad & p(0)&\in& \partial^{-,P} T(x_0).
\end{array}\right.
\end{equation*}
Then, there exist constants $c, r>0$ such that, for all
$t\in[0,T(x_0))$ and $h \in B (0,r)$,
\begin{equation*}
T(\overline{x}(t)+h)- T (\overline{x}(t)) \geq\  \langle \overline{p}(t),h\rangle  - c \mid h \mid^2.
\end{equation*}
Consequently, 
\begin{equation}
\overline{p}(t)\in \partial^{-,P}T(\overline{x}(t)) \quad \mbox{ for all } t\in [0,T(x_0)).
\end{equation}  
\end{theorem}
\begin{proof}
First of all, recall that $0\not\in \partial^{-,P} T(x_0)$ (see, for instance, \cite[Theorem~5.1]{MR1613909}), and so the dual arc $\overline{p}$ never vanishes on $[0,T(x_0)]$ by Remark \ref{RemarkDualArc:ch3}.
Since  $\overline{p}(0)\in \partial^{-,P}T(x_0)$, there exist $c_0,r_0>0$ such that for every $h \in B(0,r_0)$,
\begin{equation}\label{Step1:ch3}
T(x_0 + h)- T(x_0)  \geq  \langle \overline{p}(0), h \rangle -c_0 |h|^2.
\end{equation}
Fix $t\in( 0,T(x_0))$. Recall that $\overline{x}(\cdot)$ is the unique solution of the final value problem
\begin{equation*}
\left\{
\begin{array}{l}
-\dot{x}(s)= \nabla_p H(x(s),\overline{p}(s))\quad \mbox{ for all } s\in \left[ 0, t \right],\\
x(t)=\overline{x}(t).
\end{array}\right.
\end{equation*}
For all $h \in B$, let $x_h(\cdot)$ be
the solution of the system
\begin{equation*}
\left\{
\begin{array}{l}
-\dot{x}(s)= \nabla_p H(x(s),\overline{p}(s))\quad \mbox{ for all } s\in \left[ 0, t \right],\\
x(t)=\overline{x}(t)+h.
\end{array}\right.
\end{equation*}
From  the optimality of $\overline{x}(\cdot)$ and the dynamic
programming principle we deduce that
\begin{equation}\label{Mah:ch3}
\begin{split}
T(\overline{x}(t)+h)-&T(\overline{x}(t)) - \langle \overline{p}(t),h \rangle =
T(x_h(t))-T(\overline{x}(t))- \langle \overline{p}(t),h \rangle \\
&\geq T(x_h(0))- T(x_0)- \langle \overline{p}(t),h \rangle .
\end{split}
\end{equation}
From the sublinearity of $F$ and $(H)(ii)$, using a standard argument based on Gronwall's lemma, one can show that there exists $k>0$,
independent of $t\in (0,T(x_0))$, such that
\begin{equation}\label{StimaVecchia:ch3}
\| x_h- \overline{x} \|_{\infty}  \leq e^{k T} \mid h \mid, \quad
\forall~ h \in B .
\end{equation}
For all $h \in B(0,r)$ with $r := \min \lbrace 1, r_0 e^{-k T}\rbrace $, by \eqref{Step1:ch3}, \eqref{Mah:ch3} and \eqref{StimaVecchia:ch3} we have that
\begin{equation}\label{step2}
T(\overline{x}(t)+h)-T(\overline{x}(t)) - \langle \overline{p}(t),h \rangle \geq - \langle \overline{p}(t),h \rangle + \langle \overline{p}(0), x_h(0)- x_0 \rangle  - c_0 \mid x_h(0)-x_0 \mid^2. 
\end{equation}
Moreover, 
\[
\langle \overline{p}(0),x_h(0)- x_0 \rangle - \langle \overline{p}(t),h \rangle = - \int_{0}^t \frac{d}{d s} \langle \overline{p}(s),  x_h(s) - \overline{x}(s) \rangle\ ds
\]
\[
= - \int_{0}^t \Big( \langle \dot{\overline{p}}(s), x_h(s)- \overline{x}(s)\rangle + \langle \overline{p}(s),\dot{x}_h(s)-\dot{\overline{x}}(s) \rangle \Big) d s
\]
\[
= - \int_{0}^t \Big( \langle \dot{\overline{p}}(s), x_h(s) - \overline{x}(s)\rangle - H(x_h(s),\overline{p}(s))+ H(\overline{x}(s),\overline{p}(s)) \Big) d s.
\]
Since $\dot{\overline{p}}(s)\in \partial_x^-
H(\overline{x}(s),\overline{p}(s))$ a.e. in $[0,T(x_0)]$, assumption
(H)$\,(i)$ implies that
\begin{equation}\label{concluu:ch3}
\langle \overline{p}(0),x_h(0)- x_0 \rangle - \langle \overline{p}(t),h\rangle \geq c_2 \int_{0}^t  \mid \overline{p}(s)\mid \mid x_h(s)- \overline{x}(s)\mid^2,
\end{equation}
where $c_{2}$ is a suitable constant independent from $t \in (0,T(x_0))$. From \eqref{StimaVecchia:ch3}-\eqref{concluu:ch3} we obtain our conclusion.
\end{proof}
\begin{theorem}\label{Lemma_sub}
Assume $(SH)$ and $(H)$ hold. Let $x_0 \in\mathcal{R}\setminus \mathcal{K}$, let $\overline{x}$ be an optimal trajectory starting from $x_0$ reaching $\mathcal{K}$ in time $T(x_0)$, and $\overline{p}:[0,T(x_0)] \rightarrow \mathbb{R}^n$ be an
 arc such that $(\overline{x},\overline{p})$ solves the system: for a.e. $t\in \left[ 0, T(x_0) \right]$,
\begin{equation}\label{sub}
\left\{\begin{array}{rllrrl}
-\dot{x}(t) &\in& \partial_p^- H(x(t),p(t)), \quad & x(0)&=&x_0 \\
\dot{p}(t) &\in& \partial_x^- H(x(t),p(t)), \quad & p(0)&\in& \partial^{-} T(x_0).
\end{array}\right.
\end{equation}
Then, $\overline{p}(\cdot)$ satisfies
\begin{equation}
\overline{p}(t)\in \partial^{-}T(\overline{x}(t)) \quad \mbox{ for all } t\in [0,T(x_0)).
\end{equation}  
\end{theorem}
\begin{proof}
The proof of the case $\overline{p}\neq 0$ is similar to the proof of the above theorem. As the case $\overline{p} = 0$, we refer the reader to the proof of \cite[Theorem~2.1]{Nostro}, where we have described a strategy for constructing perturbations of the optimal trajectory when the dual arc is vanishing.
\end{proof}

\subsection{First application: differentiability of the minimum time function}
Here, we provide some sufficient conditions in order that the optimal trajectories starting from a point in the domain of differentiability of the minimum time function $T(\cdot)$ stay in such a set whenever Petrov's condition is satisfied at the final point of the optimal trajectories. 
The same result has been obtained in \cite{frankowska:hal-00981788} in the case of smooth targets and smooth control systems with a Hamiltonian of class $C^{1,1}$. Let us also mention that the fact that $T(\cdot)$ is differentiable along an optimal trajectory starting from $x$ for all time in the open interval $(0,T(x))$ has been proved earlier in \cite{MR1780579} in the case of exit-time problems with smooth control systems and a strongly convex Hamiltonian in $p$, under the Petrov condition $(PC)$. More recently, in \cite{MR3005035} this result was extended to the case of differential inclusions whit a strictly quasi-convex Hamiltonian in $p$. On the other hand, in our context the Hamiltonian is no longer strictly convex being $1$-homogeneous  in $p$ and, in general, is not strictly quasi-convex, as shown by \cite[Example 1.]{MR3005035}. 
Here, the Hamiltonian is only assumed to satisfy $(H)$, while the target is supposed to satisfy an interior sphere property.
\begin{theorem}\label{differentiability}
Assume that $(SH)$ and $(H)$ hold. Let $x_0 \in \mathcal{R}\setminus \mathcal{K}$, let $\overline{x}$ be an optimal trajectory starting from $x_0$ reaching the target $\mathcal{K}$ at time $T(x_0)$. Let $R>0$, and suppose that 
$\mathcal{K}$ is a nonempty subset of $\mathbb{R}^n$ and satisfies the inner sphere property of radius $R$.
Suppose, moreover, that $H(\overline{x}(T),\nu)>0$ for any $\nu\in N^{P}_{\mathcal{K}}(\overline{x}(T))$, $\nu \neq 0$, and $T$ is differentiable at $x_0$. 
Then, we have that $T$ is differentiable at $\overline{x}(t)$ for all $t\in [0,T(x_0))$. 
\end{theorem}
\begin{proof}
Showing that $T$ is differentiable at a certain point $z$ is equivalent to prove that $\partial^+ T(z)$ and $\partial^- T(z)$ are both nonempty. 
Thus, the conclusion in the case $(b)$ comes from Theorem \ref{Lemma_sub} together with 
\cite[Theorem~4.1]{MR3318195}. 
\end{proof}
\subsection{Second application: local $C^2$ regularity of the minimum time function}\label{sub:last}
Theorem \ref{TheoremCT} and \ref{Lemma_sub_prox:ch3} apply to show that the existence of a proximal subgradient of $T(\cdot)$ at $x$ is sufficient for the local regularity of $T(\cdot)$ in a neighborhood of the optimal trajectory starting form $x$. The proof is based upon ideas from \cite{frankowska:hal-00851752,SecondoNostro}.
\begin{theorem}\label{TheoLocalReg:ch3}
Assume $(SH), (A), (PC)$ and suppose that $H$ is of class $C^2(\mathbb{R}^n \times (\mathbb{R}^n \setminus \lbrace 0 \rbrace))$.  Let $x_0\in \mathcal{R}\smallsetminus \mathcal{K}$ and let $\overline{x}$ be an optimal trajectory starting from $x_0$ reaching $\mathcal{K}$ at time $T(x_0)$. If $\partial^{-,P} T(x_0)\neq \emptyset$, then $T$ is of class $C^2$ in a neighborhood of $\overline{x}([0,T(x_0)))$.
\end{theorem}
\begin{proof}
Recall first that the minimum time function $T$ is semiconcave (see \cite{MR2918253}). Thus, it is well known that since $\partial^{-,P} T(x_0)\neq \emptyset$, $T$ must be differentiable at $x_0$. Therefore, the optimal trajectory for $x_0$ is unique, and we call it $\overline{x}$. It reaches the target at time $T(x_0)$. Set $\xi_0 := \overline{x}(T(x_0))$.\\
Thanks to assumption $(PC)$, $T$ is differentiable at $\overline{x}(t)$ for all $t\in [0,T(x_0))$ (see Theorem \ref{differentiability}). Recall that  and $- \nabla b_{\mathcal{K}}(\xi_0)$ is a proximal inner normal to $\mathcal{K}$ at $\xi_0$ with unit norm. Thus, $\overline{p}(t)=\nabla T(\overline{x}(t))$ where $\overline{p}:[0,T(x_0)]\rightarrow \mathbb{R}^n\setminus \lbrace 0\rbrace $ is such that the pair $(\overline{x},\overline{p})$ solves: for a.e. $t\in[0,T(x_0)]$, 
\begin{equation*}
\left\{\begin{array}{rll}
-\dot{x}(t)&=& \nabla_p H (x(t),p(t)),\\
\dot{p}(t)&\in & \partial_x H (x(t),p(t)),
\end{array}\right.
\quad p(T)=- \nabla b_{\mathcal{K}}(\xi_0) H(\xi_0,\nabla b_{\mathcal{K}}(\xi_0))^{-1}. 
\end{equation*} 
Set $Y(\xi_0,\cdot)= \overline{x}(T(x_0)-\cdot)$ and $P(\xi_0,\cdot)= \overline{p}(T(x_0)-\cdot)$. Thus, $(Y(\xi_0,\cdot),P(\xi_0,\cdot))$ solves \eqref{CJ*:ch3}.\\
Thanks to Theorem \ref{TheoremCT}, it is sufficient to prove that the interval $[0,T(x_0)]$  does not contain any conjugate time for $\xi_0$. Let us proceed by contradiction, assuming that there exists a conjugate time $\overline{t}$ for $\xi_0$ with $\overline{t}\in (0 ,T(x_0))$. Fix $t\in (0,\overline{t})$. By Theorem \ref{TheoremCT}, we deduce that there exists an open neighborhood $V_{\xi_0}\times U_t$ of $(\xi_0,t)$ in $\partial \mathcal{K}\times \mathbb{R}$ such that the function $T(\cdot)$ is of class $C^2$ in the open neighborhood $Y(V_{\xi_0},U_t)$ of $Y(\xi_0,t)$. Furthermore, the Hessian of $T$ is 
\begin{equation}\label{hessian} 
D^2 T(Y( \xi,s))= R_{\xi,t}(\xi,s):= P_{\xi,t}(\xi,s)Y_{\xi,t}^{-1}(\xi,s),~ (\xi,s)\in V_{\xi_0}\times U_t,
\end{equation}
where $(P_{\xi,t},Y_{\xi,t})$ is the solution to \eqref{CP_prem:ch3}. Recall that equality \eqref{hessian} has to be understood in the sense that we have explained in Remark \ref{RemarkRapresentation}. Thus, for all $(\xi,s)\in V_{\xi_0}\times U_t$ we have that
\begin{equation}\label{a}
\begin{split}
&T(Y(\xi,s) )- T(Y(\xi_0,t)) -\langle \nabla T(Y(\xi_0,t)), Y(\xi,s) - Y(\xi_0,t)\rangle \\
&-\frac{1}{2} \langle R_{\xi,t}(\xi_0,t)(Y(\xi,s) - Y(\xi_0,t)),Y(\xi,s) - Y(\xi_0,t)\rangle = o(| Y(\xi,s) - Y(\xi_0,t)|^2).
\end{split}
\end{equation}
Moreover, since $-\overline{p}(0) \in \partial^{-,P} T(x_0)$, by Theorem \ref{Lemma_sub_prox:ch3} there exists $R_0>0$ and $c_0\geq 0$ such that 
\begin{equation}\label{b}
T(y)- T(Y(\xi_0,t)) -\langle \nabla T(Y(\xi_0,t)), y - Y(\xi_0,t)\rangle \geq -c_0 |y - Y(\xi_0,t)|^2,
\end{equation}
whenever $y\in B(Y(\xi_0,t),R_0)$ and $t\in [0,T(x_0)]$.  Without lost of generality, we can suppose that  $Y(V_{\xi_0},U_t) \subset B(Y(\xi_0,t),R_0) $. Then, by \eqref{a} and \eqref{b},
\begin{equation}\label{c}
\langle R_{\xi,t}(\xi_0,t)(Y(\xi,s) - Y(\xi_0,t)),Y(\xi,s) - Y(\xi_0,t)\rangle \geq -c_0 | Y(\xi,s) - Y(\xi_0,t)|^2,
\end{equation}
for all $(\xi,s)\in V_{\xi_0}\times U_t$. Since $Y(V_{\xi_0},U_t)$ is an open neighborhood of $ Y(\xi_0,t)$ in $\mathbb{R}^n$, from \eqref{c} we deduce that 
\begin{equation}\label{d}
\langle R_{\xi,t}(\xi_0,t)\theta,\theta\rangle \geq -c_0, \quad \forall\ \theta\in S^{n-1}.
\end{equation}
This provides a bound from below, uniform in $[0,\overline{t})$, of the quadratic form associated to $R_{\xi,t}(\xi_0,t)$. Furthermore, since $T$ is semiconcave it holds that for any $\nu \in\mathbb{R}^n$ such that $\mid \nu\mid =1$ we have $\frac{\partial^2 T}{\partial \nu^2}\leq C$ in the sense of distributions, where $C$ is the semiconcavity constant of $T$ (see, for instance, \cite[Proposition~1.1.3]{MR2041617}). Since $T$ is twice differentiable on $Y(V_{\xi_0},U_t)$ for all $t\in [0,\overline{t})$, the distributional Hessian coincides with the classical Hessian \eqref{hessian} on such sets.  We conclude that $R_{\xi,t}(\xi_0,t)$ must be bounded from above on $[0,\overline{t})$ by $C$, still in the sense of quadratic form. 
On the other hand, the operator norm of $R_{\xi,t}(\xi_0,t)$ goes to infinity as $t\rightarrow \overline{t}$, since we have supposed that $\overline{t}$ is conjugate for $\xi_0$. 
These facts together give a contradiction. 
Summarizing, we have proved that the all interval $[0,T(x_0)]$ does not contain conjugate times for $\xi_0$, and we conclude by Theorem \ref{TheoremCT} that the minimum time function $T$ is of class $C^2$ in a neighborhood of $\overline{x}([0,T(x_0)))$.
\end{proof}
\begin{example}
Here are some examples of systems with a smooth Hamiltonian. We consider a dynamic ${\dot x}= f(x,u)$ given by the \emph{control-affine system} with drift:
\begin{equation}\label{ex}
 f(x,u)= h(x)+ F(x)u, \; u\in U, 
 \end{equation}                               
where $U\subset \mathbb{R}^m$ is the closed unit ball,  $h:\mathbb{R}^n\rightarrow \mathbb{R}^n$ is the drift and
$F:\mathbb{R}^n \rightarrow \mathbb{R}^{n \times m}$ is a matrix-valued function defined by $m$ vector fields $f_i:\mathbb{R}^n \rightarrow \mathbb{R}^n$, $i=1,...m$, as $F(\cdot)=[f_1(\cdot),...,f_m(\cdot)]$. Suppose hereafter that the vector fields $h$ and $f_i$, $i=1,...,m$, are of class $C^2$ with sublinear growth.
The Hamiltonian associated to this system is $$ H(x,p)= \langle p,h(x)\rangle +  \mid F(x)^* p \mid.$$ 
We give the following two examples:
\begin{itemize}
 \item We assume that $F(x)$ is surjective for all $x\in\mathbb{R}$ (that is, $m \geq n$ and that the rank of $F(x)$ is equal to $n$ for all $x\in\mathbb{R}^n$ ). 
 Note that if $F(x)$ is surjective, then $F(x)^*$ is injective. So its kernel is reduced to the singleton set $\lbrace 0 \rbrace$.
 Thus, for all $x,p\in\mathbb{R}^n$, $p\neq 0$,
$$ H_p(x,p)= h(x)+ \frac{F(x)F^*(x)p}{|F^*(x)p |}, $$
$$ H_{pp}(x,p)= \frac{F(x)F^*(x) }{|F^*(x)p \mid } -  \frac{F(x)F(x)^* p \otimes F(x)F^*(x) p }{|F^*(x)p|^3}.$$
Thus, $F$ satisfies assumption $(SH)$ and $H$ is clearly of class $C^2 (\mathbb{R}^n
\times ( \mathbb{R}^n \setminus \lbrace 0 \rbrace)).$ Note that also hypothesis $(H2)$ is satisfied. Similarly, we can consider strictly convex sets $U$ with sufficiently smooth boundary.\\ 
Let us remark that the special case where $m=n$, $h\equiv 0$ and $F(x)$ is invertible for all $x\in\mathbb{R}^n$ corresponds to a Riemannian type problem. 
We are not in the standard Riemannian situation whenever $F(x)$ is not invertible, and in particular when  $m < n$.
We deal with subriemannnian type problems in the example below.
\item We consider the control system \eqref{ex} without drift, i.e. $h\equiv 0$.
Let us first recall the notion of singular optimal trajectory for such control system. Let us introduce the pre-Hamiltonian
$$ H_0(x,p,u)= \sum_{i=1}^m u_i \langle p_i, f_i(x)\rangle, $$
that is smooth on $\mathbb{R}^n \times \mathbb{R}^n \times U$ because of the smoothness of the fields $f_i$.
A trajectory $x:[0,T]\rightarrow\mathbb{R}^n $ reaching the target at time $T$ (and corresponding to the control $u(\cdot)$) is said to be \emph{singular} if and only if there exists an absolutely continuous arc $p:[0,T]\rightarrow \mathbb{R}^n\setminus \lbrace 0 \rbrace$ such that $(x(\cdot),p(\cdot),u(\cdot))$ solves the system
\begin{equation}\label{nonSing}
\left\{\begin{array}{rll}
-\dot{x}(t)&= \nabla_p H_0 (x(t),p(t),u(t)),\\
\dot{p}(t) &= \nabla_x H_0 (x(t),p(t),u(t)),
\end{array}\right.
\quad t\in [0,T],
\end{equation}
and $p(\cdot)$ is orthogonal to each vector $f_1(x(\cdot)),...,f_m(x(\cdot))$ on $[0,T]$ (i.e. $p(t)\in \ker F(x(t))^*$ for all $t\in [0,T]$). The control $u$ is said to be singular as well. \\
Let $\overline{x}(\cdot)$ be a nonsingular optimal trajectory and let $\overline{p}(\cdot)$ a dual arc solving \eqref{nonSing} 
and such that $\overline{p}(s)\not\in \ker F(\overline{x}(s))^*$ for a time $s$ in the interval $[0,T]$. 
Thus, even if $F(\cdot)$ is not surjective, in a neighborhood of $(\overline{x}(\cdot),\overline{p}(\cdot))$ the Hamiltonian $H$ satisfies assumption $(H)$ and is of class $C^2$, and $(\overline{x}(\cdot),\overline{p}(\cdot))$ solves also the system
\begin{equation}
\left\{\begin{array}{rll}
-\dot{x}(t)&= \nabla_p H (x(t),p(t)),\\
\dot{p}(t) &= \nabla_x H (x(t),p(t)),
\end{array}\right.
\quad t\in [0,T].
\end{equation}
To conclude, we recall that a sufficient condition to exclude the presence of nontrivial singular control 
is that the distribution corresponding to $f_1,...,f_m$ is \emph{fat}, i.e., the control system is \emph{strongly bracket generating}. 
There is a vast literature devoted to this topic.
The interested reader is referred to, e.g., \cite{ MR3308372,MR1867362} for more details on geometric control.
\end{itemize}
\end{example}

\section*{Acknowledgement}
This research is partially supported by the European Commission (FP7-PEOPLE-2010-ITN, Grant Agreement no.\ 264735-SADCO), and by the INdAM National Group GNAMPA. This work was completed while the first author was visiting the Institut Henri Poincar\'e and Institut des Hautes \'Etudes Scientifiques on a senior  CARMIN position. The authors are grateful to the anonymous referee for her/his useful comments.

\end{document}